\documentclass[11pt,a4paper]{amsart}

\usepackage{enumerate,color}
\usepackage{mdwlist}
\usepackage{mathrsfs}
\usepackage{url}

\usepackage[arrow, matrix, curve]{xy}

\usepackage[american]{babel}

\usepackage{paralist}               	
\usepackage[parfill]{parskip}			
\usepackage{a4wide}
\usepackage{amsmath,amsthm}	
\usepackage{amssymb}
\usepackage[T1]{fontenc}

\usepackage[bottom]{footmisc}			
\usepackage{bbm, dsfont}  
\renewcommand{\1}{\mathbbm{1}}
\newcommand{\mk}{\mathfrak}
\newcommand{\mc}{\mathcal}
\newcommand{\ms}{\mathscr}

\newcommand{\mf}{\mathbf}

\newcommand{\en}{\enspace}
\newcommand{\sk}[1]{\left\langle #1 \right\rangle}

\def \A {\mathfrak{A}}

\def \L {\mathscr{L}}

\def \Ta {T_{\alpha}}

\def \N {\mathbb{N}}
\def \C {\mathbb{C}}

\def \S {\mathscr{S}}
\def \SS {\mathbf{S}}
\def \GG {\mathbf{ G }}  
\def \s {\sigma}
\def \ol {\overline}

\def \a {\alpha}
\def \b {\beta}

\def \ph {\varphi}

\def \Hphi {H_{\varphi}}
\def \Aphi {\A_{\varphi}}
\def \Tphi {T_{\varphi}}

\def \ran {\operatorname{ran}}
\def \lin {\operatorname{lin}}
\def \ker {\operatorname{ker}}
\def \co {\operatorname{co}}
\def \Sp {\operatorname{Sp}}
\def \Fix {\operatorname{Fix}}

\newcommand{\WA}{$W^{*}$-algebra\ }
\newcommand{\BA}{ \mathfrak{A} }

\theoremstyle{plain}
\newtheorem*{thm*}{Theorem}
\newtheorem{thm}{Theorem}[section]
\newtheorem{lemma}[thm]{Lemma}
\newtheorem{prop}[thm]{Proposition}
\newtheorem{cor}[thm]{Corollary}
\theoremstyle{definition}

\newtheorem{remark}[thm]{Remark}

\title{Decomposition of operator semigroups on W*-algebras}
\author[A. B\'{a}tkai]{Andr\'{a}s B\'{a}tkai}
\address{ELTE TTK, Institute of Mathematics\newline 1117 Budapest, P\'{a}zm\'{a}ny P. s\'{e}t\'{a}ny 1/C, Hungary.}
\email{batka@cs.elte.hu}
\author[U. Groh]{Ulrich Groh}
\address{Institute of Mathematics\\University of T\"ubingen\\Auf der
  Morgenstelle 10\\72076 T\"ubingen\\Germany}
\email{ulgr@fa.uni-tuebingen.de}
\author[D. Kunszenti-Kov\'{a}cs]{D\'{a}vid Kunszenti-Kov\'{a}cs}
\address{Institute of Mathematics\\University of T\"ubingen\\Auf der
  Morgenstelle 10\\72076 T\"ubingen\\Germany}
\email{daku@fa.uni-tuebingen.de}
\author[M. Schreiber]{Marco Schreiber}
\address{Institute of Mathematics\\University of T\"ubingen\\Auf der
  Morgenstelle 10\\72076 T\"ubingen\\Germany}
\email{masc@fa.uni-tuebingen.de}

\date\today
\begin{document}

\begin{abstract}
We consider semigroups of operators on a W$^*$-algebra and prove, under appropriate assumptions, the existence of a Jacobs-DeLeeuw-Glicksberg type decomposition.
This decomposition splits the algebra into a ``stable'' and ``reversible'' part with respect to the semigroup and yields, among others, a structural approach to the Perron-Frobenius spectral theory for completely positive operators on W$^*$-algebras.
\end{abstract}
\maketitle

\section{Introduction}

The idea that a dynamical system (or, equivalently, the state space on which it is acting) can be decomposed into a ``structured'' and a ``random'' part is a frequent theme in ergodic theory. 
We refer to Krengel \cite[Section 2.3]{krengel} or to the recent exposition by Tao \cite{tao_blog2}, where such ``splittings'' are considered in an even more general context.

In a purely operator theoretical context, one of the first structural theorems of this type has been obtained by Jacobs \cite{jacobs57}, and has been subsequently extended and clarified by 
DeLeeuw and Glicksberg \cite{dlg59},\cite{dlg61}.

\begin{thm}\label{theo:JDLG}
Let $H$ be a Hilbert space and $\S\subset \L(H)$ a semigroup of contractions, and let $\ol{\S}$ denote its closure in the weak operator topology.
Then $H$ can be decomposed into $\S$-invariant subspaces as %
\begin{equation*}
H = H _{ \mathrm{r} } \oplus  H _{ \mathrm{s} },
\end{equation*}
where $H _{ \mathrm{r} }$ is the space of $\ol{\S}$-reversible elements, i.e.,
\begin{align*}
H _{ \mathrm{r} } & =
	\{x\in H \colon \, \text{ for all } S \in \ol{\S} \,
		\text{ there exists } \, T \in \ol{\S} \, \text{ such that } TSx=x\} \\
	& = \{x\in H \colon \| Sx \| = \| x \| \, \text{ for all } \, S\in\ol{\S} \}
\end{align*}
%
%
and $H _{ \mathrm{s} }$ consists of all elements for which $0$ is a weak accumulation point of the orbit, i.e.,
\begin{equation*}
H _{ \mathrm{s} }  =
	 \left\{ x \in H \colon 0 \in
	 		 \overline{\{ S x \colon S \in   \S \}}  ^ {\sigma ( H , H ^{ * } ) }  \right\}
=\bigcup_{T\in\overline{\S}}\ker T
.
\end{equation*}
\end{thm}

The subspaces $H _{ \mathrm{r} }$ and $ H _{ \mathrm{s} } $ are usually  referred  to as 
the \textit{reversible} and \textit{stable} subspace of the dynamical system corresponding to the semigroup 
$ \S $.

The proof of Theorem \ref{theo:JDLG} can be obtained via the theory of \emph{compact semitopological semigroups} (see e.g. the monograph \cite{Berglund}).
We briefly summarize the main idea and refer to the original papers.

Let $ \S $ be a semigroup of continuous linear operators on a Banach space $ X $.
If $ \{ S x \colon S \in \S  \} $ is  relatively $ \sigma ( X , X ^{ * } )$-compact for all $ x \in X $, 
then the closure $ \SS := \overline{\S} $ in the weak operator topology is a compact semitopological semigroup. 
It contains a minimal ideal $M(\SS)$, which itself is a semitopological semigroup 
(see DeLeeuw and Glicksberg \cite[Theorem 2.3]{dlg61} for details).

If the minimal ideal is actually a group, then its unit $P$ is a projection, i.e. $P^2=P$, and we have that
\[
	M ( \SS ) = \{ P T  P \colon T \in \SS \} = P \SS P.
\]
The theorem of Ellis \cite[Theorem 2]{ellis} shows that this compact semitopological group is already a topological group. The projection $P$ induces a decomposition of the space $ X $ into the image $ X _{ \mathrm{r} } =\ran P$ and the kernel $ X _{ \mathrm{s} } = \ker P $.
These closed subspaces have the desired properties mentioned in Theorem \ref{theo:JDLG}.
In particular, $ \SS $ acts on $ X_r $ as a compact group of operators and 
$ \SS _{ | X _{ r } } \simeq M ( \SS ) $.

We call this decomposition the \emph{Jacobs-DeLeeuw-Glicksberg decomposition} or \emph{JDLG decomposition} for short. 
It is among the most useful tools in the study of the asymptotics of operators and operator semigroups on Banach spaces.
We refer to Engel and Nagel \cite[Section V.2]{engel-nagel} for applications to evolution equations and to Krengel \cite[Section 2.4]{krengel} for further motivation and references.

In order to use the approach sketched above one needs a topology for which $\SS$ is compact and the multiplication in $ \SS $ is separately continuous.
In general,  the minimal ideal becomes a group only under additional conditions either on the semigroup, e.g. amenability as in \cite[Corollary 2.9]{dlg61}, or on the geometry of the underlying Banach space, e.g. being a Hilbert space as in \cite[Satz 5.1]{jacobs57}.

For operator semigroups on W$^{ * }$-algebras we show that such a decomposition exists under natural assumptions.
The following theorem is the main result of this paper.
\begin{thm}\label{thm:main}
Let $\S\subset \L(\A,\s^*)$ be a bounded semigroup of $\s^*$-continuous operators on a 
 W$^{ * }$-algebra $\A$ and let $ \Phi \subset \A_*$ be a faithful family of normal states on $\A$ satisfying
\begin{equation}
 \label{eq:ungleichung}
\sk{\varphi,(Sx)^*(Sx)}\le \sk{\varphi,x^*x}\qquad \text{ for all } \;  x\in\A, S\in\S, \ph\in\Phi.
\end{equation}
Then $\SS:=\ol{\S}^{\L _{ s } (   \mathfrak{A} ,  \sigma ^{ * }   )}$ is a compact semitopological semigroup containing a unique minimal ideal $M(\SS)$ which is a compact topological group with unit $P$.

This projection $ P $ induces a decomposition 
$ \mathfrak{A} = \mathfrak{A} _{ r } \oplus  \mathfrak{A} _{ s } $, where
\begin{align*}
	\ran P=\mathfrak{A} _{ r }	&=\{x\in\A:  \, \text{ for all } S \in\SS \, 
			\text{ there exists } \, T \in \SS \, \text{ such that } TSx=x\} \\
						&= \{ x \in \A \colon  \sk{\ph,(Tx)^*(Tx)}=\sk{\ph,x^*x}\en \, 
			\text{ for all } \,  T\in\SS,\ph\in\Phi\} \\
			&= \{ x \in \A \colon  \sk{\ph,(Tx)^*(Ty)}=\sk{\ph,x^*y}\en \, 
			\text{ for all $ T\in\SS,\ph\in\Phi, y \in \BA $} \}
\end{align*}
and 
\begin{equation*}
	\ker P=\mathfrak{A} _{ s } = \{x\in\A \colon  0\in \SS x\} %
		= \left\{ x \in \BA \colon 0 \in 
				\overline{\{ T x \colon T \in \SS \}} ^{ \sigma ( \BA , \BA_{*} ) } \; \right\} %
		=  \bigcup _ { T \in \SS }\ker T .
\end{equation*}
\end{thm}

The proof relies on the \emph{Gelfand-Naimark-Segal} construction  allowing to extend our semigroup to a semigroup of contractions on a Hilbert space, so that Theorem \ref{theo:JDLG} is
applicable. 

An important consequence of Theorem \ref{thm:main} is the following.

\begin{cor}
\label{cor:star_automorphisms}
If, under the assumptions of Theorem \ref{thm:main}, the operators 
in $ \S $ are completely positive and $S^*\ph=\ph$ for all $\ph\in\Phi$ and $S\in\S$, then $ \BA _{ r } $ is a W$^*$-subalgebra of $\A$ and the minimal ideal $ M ( \SS ) $ acts as a compact group of $ ^{ * }$-automorphisms on $\mathfrak{A} _{ r } $.
\end{cor}

In this case the structure of $ \ran P $ has been studied in \cite{Albeverio} and in \cite{olesen1980} for abelian semigroups $ \S $.
The non-commutative case has been investigated by M. Landstad in \cite{landstad1992} and by A. Wassermann in \cite{wassermann1989}.
For further details we refer to those papers.

In Section \ref{sec:prelim}, we summarize the necessary background on W$^{*}$-algebras.
For the convenience of the reader, we prove our main decomposition theorem in two steps. 
In Section \ref{sect:spec}, we deal with a special case where the main idea of our proof becomes evident. 
In Section \ref{sect:gen} the general case is  treated.
In Section \ref{sect:char} we give a more detailed description of the invariant subspaces arising in the decomposition.
We close this paper by giving some examples in which additional assumptions on the semigroup allow for better structural results on the decomposition.
Finally, we show in Section \ref{sect:ex} how the Perron-Frobenius theory for completely positive operators on 
W$^{*}$-algebras (see \cite{ug1984}) follows naturally from the JDLG decomposition for such operators.

\section{Terminology and basic tools}\label{sec:prelim}

For the theory of W$^{ * }$-algebras we use the notation and terminology of Takesaki \cite{Takesaki} and Sakai \cite{Sakai}.

By $  \mathfrak{A} $ we always denote a  W$^{*}$-algebra with unit $ \1 $ and predual $  \mathfrak{A} _{ * } $ and its  positive cone $  \mathfrak{A} ^{ + } $.
By $ \sigma ^ { * } $ (resp.\@ $\sigma $) we denote the
$ \sigma (  \mathfrak{A} ,  \mathfrak{A} _{ * } ) $-topology on $  \mathfrak{A} $ (resp.\@ the
$ \sigma (  \mathfrak{A} _{ * },  \mathfrak{A} ) $-topology on $  \mathfrak{A} _{ * } $).

A normal linear form $ \ph \in  \mathfrak{A} _{ * } $ is positive if
$ \sk{\ph,  x } := \ph(x) \geq 0 $ for all $  x \in \mathfrak{A} ^{ + } $ and  
$  \mathfrak{A} _{ * } ^{ + } $ stands for the cone of all positive normal linear forms. Further, if 
$ \sk{\ph, \1 } = 1 $, then $ \ph $ is called a (normal) state.


If $ \ph \in \mathfrak{A} _{ * } ^{ + } $, then
$ \mathrm { s }  _ {  \ph  } ( x ) := \sk{\ph,  x ^{ * } x } ^{ 1 / 2 } $ defines a seminorm on $  \mathfrak{A} $.
The topology generated by the family $ \{ \mathrm { s }  _ {  \ph  } \colon \ph \in \mathfrak{A} _{ * } ^{ + } \} $ 
is called the $ s (  \mathfrak{A} ,  \mathfrak{A} _{ * } )$-topology or $s$-topology.

A set $ \Phi \subseteq \mathfrak{A} _{ * } ^{ + } $ is called {faithful} if
$  x \in \mathfrak{A} ^{ + } $ and $ \sk{\ph, x} = 0 $ for all $ \ph \in \Phi $ implies
$ x = 0 $.

We use the following notations for the various operator topologies on $\L(\A)$, the space of all bounded linear operators on $\A$:
\begin{itemize}
\item
The \emph{weak$ ^{ * } $ operator topology} on $ \L (  \mathfrak{A} ) $ is
the topology of pointwise convergence on  $ \mathfrak{A} $ endowed with the 
$ \sigma ( \mathfrak{A} , \mathfrak{A} _{ * } ) $-topology.
We denote it by 
$ \L _{ s } (   \mathfrak{A} ,  \sigma ^{ * }   )$.

\item
The \emph{weak operator topology} on $ \L (  \mathfrak{A} _{ * } ) $ is
the topology of pointwise convergence on  $ \mathfrak{A}_{ * } $ endowed with the 
$ \sigma ( \mathfrak{A} _{ * } , \mathfrak{A}  ) $-topology.
We denote it by 
$ \L _{ s } (  \mathfrak{A} _{ * } , \sigma ) $.

\end{itemize}

The following two results will be used throughout the paper.

\begin{lemma}\label{lemma2}
Let $ \S $ be a subsemigroup of the space $ \L ( \A, \s^* ) $ of all $\s^*$-continuous operators on $\A$. Then $ \S $ is semitopological, i.e., the multiplication is separately continuous in 
$ \L _{ s } (   \mathfrak{A} ,  \sigma ^{ * }   ) $.
\end{lemma}
\begin{proof}
Let $(\Ta)\subset \S$ be a convergent net such that $\Ta\xrightarrow []{\L _{ s } (   \mathfrak{A} ,  \sigma ^{ * }   )} T\in\S$ and take $S\in \S$. It is easy to see that multiplication from the right is continuous, so we only consider multiplication from the left.
Since $ S \in \L (  \mathfrak{A} ) $ is continuous on $  \mathfrak{A} $ for the
$ \sigma ^{ * } $-topology, its adjoint $ S ^{ * } $  on $  \mathfrak{A} ^{ *} $ leaves
the predual invariant.
Thus $ S $ possesses a pre-adjoint $  S _{ * } :=S^*|_{\mk{A}}$ on $  \mathfrak{A} _{ * } $, i.e.,
$ S = (  S _{ * } ) ^{ * } $.

Hence for all $x\in\A, \psi\in\A_*$ we obtain that
$$\sk{S\Ta x,\psi}=\sk{\Ta x, S_*\psi}\longrightarrow \sk{Tx,S_*\psi}=\sk{STx,\psi}.$$
\end{proof}
%

For the proof of the following lemma see Sakai \cite[Thm. 1.7.8]{Sakai}.
\begin{lemma}
 \label{lemma:Sakai_1.7.8}
The mappings $x\mapsto x^*$, $x\mapsto yx$ and $x\mapsto xy$ are continuous for the $\s^*$-topology for all $y\in\A$.
\end{lemma}

\section{Case of a single faithful normal state}\label{sect:spec}

We now treat the special case where the family $\varPhi$ consists of a single element $\varphi$.

If $\ph\in\A_*$ is a faithful normal state on $\A$, then the sesquilinear form $$\sk{x,y}_{\ph}:=\sk{\ph,y^*x}\quad \text{ for all } \;  x,y\in\A$$
turns $\mk{A}$ into a pre-Hilbert space $(\A_{\ph},\sk{\cdot,\cdot}_{\ph})$. We denote its completion by $(\Hphi,\sk{\cdot,\cdot}_{\Hphi})$.
Note that we have the following correspondence between the $\|\cdot\|_{\ph}$- and $s(\A,\A_*)$-topologies, allowing us to apply Hilbert space results in our context.

\begin{lemma}\label{lem:top1}
Let $  \mathfrak{A} _{ 1 } $ be the unit ball of $  \mathfrak{A} $ and let $ \ph $ be a faithful
normal
state on $  \mathfrak{A} $.
Then the $ s $-topology and the topology generated by
$ \|\cdot\|_\ph $ coincide on $  \mathfrak{A} _{ 1 } $.
\end{lemma}
\begin{proof}
From Takesaki \cite[III.5.3]{Takesaki} it follows that the $ \| \cdot \|_{ \ph } $-topology is equivalent on $ \mk{A}_{ 1 } $ to the $ \sigma$-strong operator topology (which is sometimes called the strongest operator topology).
By Sakai \cite[1.15.2]{Sakai} this topology is equivalent to the $ s ( \mk{A} , \mk{A} _{ * } )$-topology on
$ \mk{A} _{ 1 } $.
\end{proof}

We now proceed to the main result of this section.

\begin{thm}\label{thm:spec}
 Let $\S\subset \L(\A,\s^*)$ be a bounded semigroup of $\s^*$-continuous operators on $\A$ and let $\ph\in\A_*$ be a faithful normal state on $\A$ satisfying
\begin{equation}
 \label{eq:ungleichung2}
\sk{\varphi,(Sx)^*(Sx)}\le \sk{\varphi,x^*x}\qquad \text{ for all } x\in\A, S\in\S.
\end{equation}
Then $\SS:=\ol{\S}^{\L _{ s } (   \mathfrak{A} ,  \sigma ^{ * }   )}$ is a compact semitopological semigroup whose minimal ideal $M(\SS)$ is a compact topological group. The unit $P$ of this group is the unique minimal projection of $\SS$.
\end{thm}

\begin{proof}
We first show that
\begin{equation}\label{eq:ungleichungGen}
\sk{\varphi,(Tx)^*(Tx)}\le\sk{\varphi,x^*x}
\end{equation}
holds for all $x\in\A$ and $T\in\SS$.
Let $x\in\A$ and $T\in\SS$. Then there is a net $(T_{\a})\subset\S$ with $\Ta\to T$ in the weak$^*$ operator topology. By Lemma \ref{lemma:Sakai_1.7.8} we have
\begin{align*}
 (\Ta x)^*&\xrightarrow[]{\s^*}(Tx)^*,\\
 (\Ta x)^*(Tx)&\xrightarrow[]{\s^*}(Tx)^*(Tx), \text{ and}\\
(Tx)^*(\Ta x)&\xrightarrow[]{\s^*}(Tx)^*(Tx).
\end{align*}
Hence,
\begin{align*}
 0&\le \sk{\ph,(\Ta x-Tx)^*(\Ta x-Tx)}\\
&=\sk{\ph,(\Ta x)^*(\Ta x)}-\sk{\ph,(\Ta x)^*(Tx)}-\sk{\ph,(T x)^*(\Ta x)}+\sk{\ph,(Tx)^*(Tx)}\\
&\le\sk{\ph,x^*x}-\sk{\ph,(\Ta x)^*(Tx)}-\sk{\ph,(T x)^*(\Ta x)}+\sk{\ph,(Tx)^*(Tx)}\\
&\to\sk{\ph,x^*x}-\sk{\ph,(T x)^*(Tx)}-\sk{\ph,(T x)^*(T x)}+\sk{\ph,(Tx)^*(Tx)}\\
&=\sk{\ph,x^*x}-\sk{\ph,(T x)^*(Tx)},
\end{align*}
and thus
$$\sk{\ph,(T x)^*(Tx)}\le\sk{\ph,x^*x}.$$

Denoting by $i:\A_{\ph}\hookrightarrow\Hphi$ the natural injection map, the extensions $\Tphi\in\L(\Hphi)$ of the operators $T\in\L(\A)$ make the following diagram commutative.

\begin{center}
\quad\begin{xy}
 \xymatrix{
     \Hphi \ar[r]^{\Tphi}    &   \Hphi   \\
     \Aphi \ar[r]^T   \ar[u]^i           &   \Aphi   \ar[u]^i
 }
\end{xy}
 \end{center}

By inequality (\ref{eq:ungleichungGen}), the semigroup $\SS$ extends to a contraction semigroup $\SS_{\ph}:=\{S_\ph: S\in\SS\}$ on $\Hphi$, i.e.,
$$\|S_\ph \xi\|_{\Hphi}\le \|\xi\|_{\Hphi}\quad \text{ for all } \;  \xi\in\Hphi,\; S\in\SS.$$

Second, we show that endowing $\ms{L}(H_\ph)$ with the weak operator topology, the mapping $\theta_{\varphi}:\SS\to\SS_{\ph}$, $S\mapsto S_\ph$ turns into a continuous semigroup isomorphism.

The only non-obvious statement is the continuity of the mapping $\theta_\ph$. 
To this end, let 
$ \Ta \xrightarrow []{\L _{ s } (   \mathfrak{A} ,  \sigma ^{ * }   )}T$ and $x,y\in\Aphi$. By Lemma \ref{lemma:Sakai_1.7.8} we obtain $y^*\Ta x\xrightarrow []{\s^*} y^* Tx$ and thus
 \begin{align*}
\sk{\theta_{\varphi}(\Ta)i(x),i(y)}_{\Hphi}&=\sk{\Ta x,y}_{\ph}=\sk{\varphi,y^*\Ta x}_{\A_*,\A}\\
&\to \sk{\varphi,y^*T x}_{\A_*,\A}=\sk{T x,y}_{\ph}=\sk{\theta_{\varphi}(T)i(x),i(y)}_{\Hphi}.
 \end{align*}

The norm boundedness of $\SS_{\varphi}$ and the density of $i(\Aphi)$ in $\Hphi$ imply
$$\sk{\theta_{\varphi}(\Ta)\xi,\zeta}_{\Hphi}\to \sk{\theta_{\varphi}(T)\xi,\zeta}_{\Hphi}$$
for all $\xi,\zeta\in\Hphi$, yielding the desired continuity.

Third, we show that $\SS$ is a subsemigroup of $\L(\A,\s^*)$. Let $(T_{\a})_{\a\in A}\subset\S$ be a net with $T_{\a}\xrightarrow []{\L _{ s } (   \mathfrak{A} ,  \sigma ^{ * }   )}T\in\L(\A)$. 
To show that $T$ is $\s^*$-continuous, it suffices to prove that $ T x_{ \b } \xrightarrow[] { \s^* } 0 $ for all decreasing nets
$(x_{\b})_{\b\in B}$ of positive elements with $\|x_\beta\| \leq 1 $ and 
$ \inf _{ \beta } x _{ \beta } = 0 $ (see Sakai \cite[Theorem 1.13.2]{Sakai}). 
Let $(x_{\b})_{\b\in B}$ be such a net. It follows from Sakai \cite[1.13.2]{Sakai} that $\inf_\beta \sk{\ph,x_\beta}=\sk{\ph,\inf_\beta x_\beta}=0$.
By (\ref{eq:ungleichungGen}) we have
\[
	\|T_{\a}x_{\b}\|^2_{\ph}=\sk{\ph,(T_{\a}x_{\b})^*(T_{\a}x_{\b})} 
		\le \sk{\ph,x_{\b}^*x_{\b}} \leq \sk{\ph,x_{\b}}\to 0
\]
uniformly in $\a\in A$. 
From Lemma \ref{lem:top1} we deduce
$$ T_{\a}x_{\b}\xrightarrow[\b]{s(\A,\A_*)}0$$
uniformly in $\a\in A$. 
Since the $s(\A,\A_*)$-topology is finer than the $\s^*$-topology, we obtain
$$ T_{\a}x_{\b}\xrightarrow[\b]{\s^*}0$$
uniformly in $\a\in A$. 

This uniform convergence together with the fact that
$$T_{\a}x_{\b}\xrightarrow[\a]{\s^*}Tx_{\b}$$
for all $\b\in B$ implies  $Tx_{\b}\xrightarrow[\b]{\s^*}0$.
Hence $\SS$ is semitopological by Lemma \ref{lemma2}.

To show that $\SS$ is compact, notice that $\ol{\S}^{\L _{ s } (   \mathfrak{A} ,  \sigma ^{ * }   )}$ is compact if and only if $\ol{\S x}^{\s^*}$ is compact for all $x\in\A$. One direction follows directly from the continuity of the mapping $T\mapsto Tx$ for every $x\in \A$ and the converse implication is a consequence of the inclusion
$$(\S,\L_s(\A,\sigma^*)\subset \prod_{x\in\A}(\ol{\S x}^{\sigma^*},\sigma^*)$$
and Tychonoff's theorem.
Since $\S$ is norm-bounded, the claim follows from the Banach-Alaoglu theorem. By DeLeeuw and Glicksberg \cite[Theorem 2.3]{dlg61}, this semigroup has a minimal ideal $M(\SS)$.

Since $\theta_{\varphi}$ is a semigroup isomorphism, it maps the minimal ideal $M(\SS)$ of $\SS$ onto the minimal ideal $M(\SS_{\ph})$ of $\SS_{\varphi}$. 
The fact that $\SS$ is compact implies that $\SS_{\varphi}=\theta_{\varphi}(\SS)$ is a compact semitopological semigroup of contractions.

Hence any minimal projection in $\SS_\ph$ is also contractive and thus orthogonal. It follows that for any minimal projection $P_\ph$ we have
\[
\ran P_\ph=\left\{\xi\in H_\ph\,:\, \|S_\ph\xi\|=\|\xi\|\,  \text{ for all } \;  S_\ph\in\SS_\ph \right\}.
\]
To see this, let $\xi=P_\ph \xi\in \ran P_\ph$ and let $S_\ph\in\SS_\ph$. Since $P_\ph$ is a minimal projection the set $P_\ph\SS_\ph P_\ph$ is a group with unit $P_\ph$. Therefore there exists some $T_\ph\in\SS_\ph$ such that $(P_\ph T_\ph P_\ph)(P_\ph S_\ph P_\ph)=P_\ph$. This implies 
$$\|\xi\|=\|P_\ph \xi\|=\|P_\ph T_\ph P_\ph S_\ph P_\ph \xi\|\le\|S_\ph \xi\|\le\|\xi\|$$
and thus $\|S_\ph \xi\|=\|\xi\|$.

Conversely, if $\|S_\ph \xi\|=\|\xi\|$ for all $S_\ph \in\SS_\ph$ then in particular $\|P_\ph\xi\|=\|\xi\|$. But this yields $\xi=P_\ph\xi$, since otherwise the strict convexity of $H_\ph$ would imply
$$\|P_\ph\xi\|=\left\|\frac{1}{2} P_\ph(\xi+P_\ph \xi)\right\|\le\left\|\frac{1}{2} (\xi+P_\ph\xi)\right\|<\|\xi\|=\|P_\ph\xi\|.$$

Hence there is a unique minimal projection, and thus the minimal ideal $M(\SS_\ph)$ is a group with unit the unique minimal projection $P_\ph$ of $\SS_\ph$. Therefore the minimal ideal $M(\SS)$ is itself a group, with unit the unique minimal projection $P$. Since $M(\SS)=P\SS P$, we have that $M(\SS)$ is compact. The theorem of Ellis 
\cite[Theorem 2]{ellis} finally implies that $M(\SS)$ is a topological group.
\end{proof}

\section{Case of a faithful family of normal states}\label{sect:gen}
In this section, let $\ms{S}\subset\ms{L}(\mk{A},\sigma^*)$ be a bounded semigroup of $\sigma^*$-continuous operators on $\mk{A}$, and let $\Phi\subset\mk{A}_*$ be a faithful family of normal states on $\mk{A}$ satisfying
\[
\begin{array}{lr}
\langle\ph,(Sx)^*(Sx)\rangle\leq\langle\ph,x^*x\rangle&  \text{ for all } \;  x\in\mk{A}, S\in\ms{S},\ph\in\Phi.
\end{array}
\]
Each $\ph\in\Phi$ induces a seminorm $x\mapsto\|x\|_\ph:=\langle\ph,x^*x\rangle^{1/2}$ on $\mk{A}$.
Let us denote by $\mc{T}_\Phi$ the topology generated by this family of seminorms.

The following is a well known generalisation of Lemma \ref{lem:top1}.
\begin{lemma}\label{lem:top2}
Let $  \mathfrak{A} _{ 1 } $ be the unit ball of $  \mathfrak{A} $ and let $ \Phi $ be a faithful
family of normal states on $  \mathfrak{A} $.
Then the $ s $-topology and 
$\mc{T}_\Phi$ coincide on $  \mathfrak{A} _{ 1 } $.
\end{lemma}
\begin{proof}
Since $ \Phi $ is a faithful family of normal states on $ \mk{A} $, for every net $x_\alpha$ in $\A$ we have that $ \| x_\alpha  \|_\ph \to 0$ if and only if $x_\alpha p _{ \ph } \to 0 $ in the strong operator topology for all $\ph\in\Phi$, see
Dixmier \cite[I.4.4]{diximier}, where $ p _{ \ph } $ is the support projection of $ \ph\in\Phi $.

Since $ \sup _{ \ph \in \Phi } p _{ \ph } = \1 $, we obtain that $ \| x_\a  \|_\ph \to 0 $ for all
$ \ph \in \Phi$ if and only if $x_\alpha \to 0 $ in the strong operator topology.

By Sakai \cite[1.15.2.]{Sakai}, this topology is  equivalent to the $ s$-topology on $ \mk{A}_{ 1 } $.
\end{proof}

Take $\ph\in\Phi$, denote its support by $p_\ph$, and the orthogonal complement of the support by $p_\ph^{\bot}:=1-p_\ph$. The space $\mk{A}$ can then be represented as the orthogonal sum $\mk{A}=K_\ph\oplus L_\ph$ where
\[
\begin{array}{lr}
K_\ph:=\mk{A}p_\ph\en\text{ and }&L_\ph:=\mk{A}p_\ph^\bot = \{ x \in \mathfrak{A} \colon \langle\ph, x ^{*}x \rangle = 0 \}
			\subseteq \ker \ph . 
\end{array}
\]

Note that $L_\ph$ is the kernel of the seminorm $\|\cdot\|_\ph$, while the same seminorm turns $K_\ph$ into a pre-Hilbert space $\left(K_\ph,\langle\cdot,\cdot\rangle_\ph\right)$ with $\langle x,y\rangle_\ph:=\langle\ph,y^*x\rangle$. The subspaces $L_\ph$ are invariant under $\SS$ since
\[
\|Sx\|_\ph^2=\langle\ph,(Sx)^*(Sx)\rangle\leq\langle\ph,x^*x\rangle=\|x\|_\ph^2=0
\]
for every $\ph\in\Phi,x\in L_\ph$ and $S\in\SS$.

A simple computation shows that for any operator $R\in\ms{L}(\A)$ leaving $L_\ph$ invariant, we have
\begin{equation}\label{eq:norm}
\|R(xp_\ph)\|_\ph=\|Rx\|_\ph
\end{equation}
for all $x\in\A$.

Denote by $\left(H_\ph,\langle\cdot,\cdot\rangle_{H_\ph}\right)$ the completion of $K_\ph$ with respect to $\|\cdot\|_\ph$. 
For each $T\in\SS$ define the operator $T_\ph:K_\ph\to K_\ph$  through
\[
T_\ph x:=(Tx)p_\ph.
\]
Note that, as in the previous section, inequality (\ref{eq:ungleichungGen}) holds for all $T\in\SS$ and $\ph\in\Phi$.
Therefore we have 
$$\|T_\ph x\|_\ph=\|(Tx)p_\ph\|_\ph=\|Tx\|_\ph\leq\|x\|_\ph,$$
implying that each $T_\ph$ is an element of $\ms{L}(K_\ph)$, and hence has a continuous extension to a contraction $\overline{T}_\ph$ on $H_\ph$. Denote by $\theta_\ph$ the mapping $T\mapsto \overline{T}_\ph$ ($T\in\SS$), and by $\SS_\ph\subset\ms{L}(H_\ph)$ its image.


Since $\Phi$ is faithful, the linear span of the spaces $(K_\ph)_{\ph\in\Phi}$ is dense in $\mk{A}$ (see Dixmier \cite[Section I.4.6]{diximier}), 
and thus the family of representations will itself be faithful. 

We now show that each 
$\theta_\ph$ is a continuous representation of the semigroup $\SS$.

\begin{prop}\label{prop:cont}
Endowing $\ms{L}(H_\ph)$ with the weak operator topology turns the mapping $\theta_\ph: \SS\to\SS_\ph$ into a continuous semigroup homomorphism.
\end{prop}
\begin{proof}
We show first that $\theta_\ph$ is  a semigroup homomorphism. Let $S, S'\in\SS$ and $x\in K_\ph$.
Then
\begin{eqnarray*}
(S_\ph S_\ph')x&=&S_\ph(S_\ph' x)=(S((S'x)p_\ph))p_\ph=(S((S'x)(1-p_\ph^\bot)))p_\ph
\\&=&(S(S'x))p_\ph-(S((S'x)p_\ph^\bot))p_\ph=(SS')_\ph x-(S((S'x)p_\ph^\bot))p_\ph.
\end{eqnarray*}

Note that $(S'x)p_\ph^\bot\in L_\ph$. By the $\SS$-invariance of $L_\ph$ we have
\[
(S((S'x)p_\ph^\bot))p_\ph\in L_\ph p_\ph=\left\{0\right\}.
\]
Thus $(S_\ph S_\ph')=(SS')_\ph$ on $K_\ph$, and by density also $(\overline{S}_\ph \overline{S'}_\ph)=(\overline{SS'})_\ph$.

For the proof of continuity, let $T_\alpha\xrightarrow{\L _{ s } (   \mathfrak{A} ,  \sigma ^{ * }   )} T$ and $x,y\in K_\ph$. By Lemma \ref{lemma:Sakai_1.7.8} we obtain $y^*(T_\alpha x)p_\ph\xrightarrow{\sigma^*} y^*(Tx)p_\ph$ and thus

\begin{eqnarray*}
\langle\theta_\ph(T_\alpha)x,y\rangle_{H_\ph}&=&\langle (T_\alpha x)p_\ph,y\rangle_\ph =\langle \ph,y^*(T_\alpha x)p_\ph\rangle_{\mk{A}_*,\mk{A}}\\
&\to&\langle\ph,y^*(Tx)p_\ph \rangle_{\mk{A}_*,\mk{A}}=\langle (Tx)p_\ph,y\rangle_\ph=\langle \theta_\ph(T)x,y\rangle_{H_\ph}.
\end{eqnarray*}

The norm boundedness of $\SS$ and the density of $K_\ph$ in $H_\ph$ implies
\[
\langle \theta_\ph(T_\alpha)\xi,\zeta\rangle_{H_\ph}\to\langle \theta_\ph(T)\xi,\zeta\rangle_{\Hphi}
\]
for all $\xi,\zeta\in H_\ph$, yielding the desired continuity.
\end{proof}

We now proceed to prove Theorem \ref{thm:main}.
\begin{proof}[Proof of Theorem \ref{thm:main}]
The proof that $\SS$ is a compact semitopological semigroup is analogous to the proof of Theorem \ref{thm:spec}. For the uniform in $\alpha$ convergence of $(T_\alpha p_\beta)_{\beta\in B}$ to $0$ in the $s(\mk{A},\mk{A}_*)$-topology, 
note that this net converges  in $\ph$-seminorm uniformly in $\alpha$ for each $\ph\in\Phi$. Hence it converges uniformly in $\alpha$ with respect to the $\mc{T}_\Phi$ topology. Lemma \ref{lem:top2} concludes the argument.

Since $\theta_\ph$ is a continuous semigroup homomorphism by Proposition \ref{prop:cont}, we obtain that $\SS_\ph$ is a compact semitopological semigroup of contractions for each $\ph\in\Phi$.
Its minimal ideal $M(\SS_\ph)$ is then a compact topological group with a unique minimal projection as unit. 
As $\theta_\ph$ is a semigroup homomorphism, we have $\theta_\ph(M(\SS))=M(\SS_\ph)$.
If $P,Q\in M(\SS)$ are two minimal projections, then the corresponding minimal projections
$\overline{P}_\ph,\overline{Q}_\ph$ in $M(\SS_\ph)$ coincide, hence $P|_{K_\ph}=Q|_{K_\ph}$ for all $\ph\in\Phi$. Since $\lin \{K_\ph\}_{\ph\in\Phi}$ is dense in $\A$, we obtain $P=Q$.
\end{proof}

As a special case we obtain the main result of K\"ummerer and Nagel \cite{kummerer-nagel}.

\begin{cor}
 Under the assumptions of Theorem \ref{thm:main}, the semigroup $\S$ is {weak* mean ergodic}, i.e., the weak* operator closure $\ol{\co\S}^{\L _{ s } (   \mathfrak{A} ,  \sigma ^{ * }   )}$ of the convex hull $\co(\S)$ of $\S$ contains a projection $P$ satisfying
\[
	TP=PT=P\quad \text{ for all } \;  T\in\S
\]
and
\[
	\ran P = \{ x \in \BA \colon T x = x \; \text{ for all $ T \in \S $ } \}.
\]
\end{cor}
\begin{proof}
If the semigroup $\S$ is bounded, consists of $\s^*$-continuous operators on $\A$ and satisfies inequality (\ref{eq:ungleichung2}), then so does the semigroup $\co\S$. By the above considerations
$\SS:= \ol{\co\S}^{\L _{ s } (   \mathfrak{A} ,  \sigma ^{ * }   )}$ is a compact semitopological semigroup and its minimal ideal $M(\SS)=\SS P=P\SS$ is a compact convex group with $P$ its unit. Hence $M(\SS)=\{P\}$ by \cite[Lemma 7.1]{dlg61}, and the result follows.
\end{proof}

\section{The Range and Kernel of $ P $}\label{sect:char}

In the situation of Theorem \ref{thm:main} we call the subspace $\A_r:=\ran P$ the \emph{reversible part} and the subspace 
$\A_s:=\ker P$ the \emph{stable part} of $\A$. 

We now describe these spaces in more detail and start with the reversible part. 


The following result is a consequence of Theorem \ref{theo:JDLG}.

\begin{prop}\label{thm:range-and-kernel}
The range of $P$ can be characterized as
  \begin{align*}
  \ran P&=\{x\in\A \colon \sk{\ph,(Tx)^*(Tx)}=\sk{\ph,x^*x}\en
 	\; \forall\; T\in\SS, \,  \ph\in\Phi   \} \\
	&= \{x\in\A \colon \sk{\ph,(Tx)^*(Ty)}=\sk{\ph,x^*y}\en
 			\; \forall \; T\in\SS, \,  \ph\in\Phi,\, y \in \BA   \}.   
   \end{align*}

\end{prop}
\begin{proof}
 Let $y=Py\in\ran P$, $T\in\SS$ and $\ph\in\Phi$. Then we can write $y=yp_\ph+yp_\ph^{\bot}$. Since $yp_\ph^\bot\in L_\ph$ and $L_\ph$ is invariant under $P$, we have
\[
yp_\ph+yp_\ph^\bot=y=Py=P(yp_\ph)+P(yp_\ph^\bot)=P(yp_\ph)+yp_\ph^\bot,
\]
hence $yp_\ph=P(yp_\ph)$. Thus $P_\ph(yp_\ph)=(P(yp_\ph))p_\ph=yp_\ph p_\ph=yp_\ph$, 
i.e., $yp_\ph\in\ran P_\ph$.
Using Theorem \ref{theo:JDLG} and the equality (\ref{eq:norm})
we obtain
\[
\langle\ph,y^*y\rangle=\langle\ph,(yp_\ph)^*yp_\ph\rangle=\langle\ph,(T(yp_\ph))^*T(yp_\ph)\rangle=\langle\ph,(Ty)^*Ty\rangle.
\]

This proves the first equality. 
The second equality follows from the Cauchy-Schwarz inequality (see e.g.\@ \cite[Corollary V.2.2]{bourbaki-tvs}).

\end{proof}

As stated in Section 1 the semigroup $\SS$ acts on $ \BA _{ r }$ as the compact topological group $\GG:=M(\SS)$ through
\begin{align*}
  T:\GG &\to \L(\A_r)\\
g&\mapsto T_g=g|_{\A_r}.
\end{align*}

Further we denote by $m$ the Haar measure of $\GG$.
 
It is a consequence of the Peter-Weyl Theorem  that the union of all finite dimensional $ \GG $-invariant subspaces of $ \BA _{ r } $ is dense in $ \BA _{ r } $ 
(see  \cite[Theorem 7.2.1]{deitmar}).

If the group $ \GG $ is abelian, then the irreducible subspaces of $ \GG $ in $ \BA _{ r } $ 
are one-dimensional and hence
\[
	\BA _{ r } = \overline{\lin} \{x\in\BA: \exists \;\hat{g}\in \widehat{\GG} 
			\text{ such that } gx=\hat{g}(g)x\en \forall \;  g\in\GG\},
\]
where the closure is taken with respect to the $ \sigma ( \BA , \BA _{ * } ) $-topology and $\widehat{\GG}$ denotes the character group of $\GG$.

We call these subspaces \emph{generalised eigenspaces} pertaining to the \emph{generalised eigenvalue} $ \hat{g} $. 

We say that the action of the group $ \GG $ on $ \BA _{ r } $ is \emph{ergodic} if the fixed space of $\GG$ is spanned by the unit of $ \BA _{ r } $.

\begin{prop}\label{prop:unitary-eigenvector-trace}
If the group $\GG$ is abelian and its action on $\A_r$ is ergodic, then $\A_r$ is generated by a family $\{u_{\hat{g}}\}_{\hat{g}\in\widehat{\mf{G}}}$ of unitary generalised eigenvectors with $\langle\ph,u_{\hat{g_1}}^*u_{\hat{g_2}}\rangle=0$ for all $\hat{g_1}\neq \hat{g_2}\in\widehat{\GG}$. Moreover, $\ph|_{\A_r}$ is a normal trace on $\A_r$ for each $\ph\in\Phi$.
\end{prop}

\begin{proof}
For a fixed $ x \in \BA $ we define
\begin{equation}\label{eq:gen-eigenvector}
	x _{ \hat{ g } } := \int _{ \GG }  \hat{g}(g) g x\, dm ( g )  
\end{equation}
 for $ \hat{g} \in \widehat{\GG} $. Then 
\[
	 h x _{ \hat{ g } } = \overline{ \hat{ g }(h) }\, x _{ \hat{ g } }
\]
for all $ h \in \GG $, i.e. $ x _{ \hat{ g } } \in \BA _{ r } $ and 
$ x _{ \hat{ g } } ^{ * } x _{ \hat{ g } } $ is a fixed vector for the group $ \GG $.

If $ x _{ \hat{ g } } \neq 0 $, then 
\[
	u _{ \hat{g}} := \|  x _{ \hat{ g } } \| ^{ - 1 } x _{ \hat{ g } }
\]
is unitary and we can assume $ \BA _{ r } $ to be generated by unitary elements (c.f. \cite[2.3]{olesen1980}).

Let  $ x, y \in \BA , x \neq y $,  
$ \hat{ g }_{ 1 } , \hat{ g }_{ 2 }   \in \widehat{ \GG}, \hat{ g }_{ 1 }   \neq \hat{ g }_{ 2 }  $
and let $ x _{  \hat{ g }_{ 1 } }, y _{ \hat{ g }_{ 2 } } $ be defined as in 
\eqref{eq:gen-eigenvector}. Then for each $\ph\in \Phi $ we obtain by Theorem \ref{thm:range-and-kernel} that
\[
	\langle\varphi , x _{  \hat{ g }_{ 1 } } ^{ * } y _ { \hat{ g }_{ 2 } } \rangle =\langle\varphi,  (g x _{  \hat{ g }_{ 1 } })  ^{ * } (g y _{ \hat{ g }_{ 2 } }) \rangle =
			  \overline{\hat{ g }_{ 1 }} \hat{ g }_{ 2 }(g)  \langle 
			\varphi , x _{  \hat{ g }_{ 1 } } ^{ * } y _ { \hat{ g }_{ 2 } }\rangle.
\]
Since $ \overline{\hat{ g }_{ 1 }} \hat{ g }_{ 2 } \neq \1  $, this implies 
$ \langle \varphi, x _{  \hat{ g }_{ 1 } } ^{ * } y _ { \hat{ g }_{ 2 } }\rangle = 0$ for all $\ph\in\Phi$.

Further, for $ x = y $,  $ \hat{g} _{ 1 } = \hat{g} _{ 2 } =: \hat{g} $ and $ x _{ \hat{g} } \neq 0 $ the unitary $ u _{ \hat{g} } =  \|  x _{ \hat{ g } } \| ^{ - 1 } x _{ \hat{g} } $ satisfies 
$\langle \varphi, u _{ \hat{g} } ^{ * } u _{ \hat{g} } \rangle = \langle\varphi , u _{ \hat{g} } u _{ \hat{g} } ^{ * }\rangle$ for all $\ph\in\Phi$. 
If $ x \in \BA_{r} $ is a linear combination of such unitaries, we have 
$ \langle \varphi, x ^{*} x \rangle = \langle\varphi, x x ^{ * }\rangle $, hence $ \varphi  $ acts as a normal trace on 
$ \BA _{ r } $ (c.f. \cite[Theorem 2.2]{stoermer}).

\end{proof}

In general $\ran P$ is devoid of any algebra structure. However, if $\S$ consists of completely positive operators then $P$ itself will also be completely positive and allows us to introduce a new multiplication turning $\ran P$ into a W$^*$-algebra.

\begin{prop}\label{prop:choi-effros}
If in addition to the hypotheses of Theorem \ref{thm:main} the semigroup $\S$ consists of completely positive operators, then $ \ran P $ is a $W^{*}$-algebra with the new product 
\begin{equation}\label{E:new-product}
	x \cdot y := P ( x y ) \, , \, x , y \in \ran P \, .
\end{equation}
\end{prop}

\begin{proof}
Choi and Effros have shown in \cite{choi1977} that 2-positivity of $P$ implies that the multiplication defined in (\ref{E:new-product}) turns $\ran P$ into a C$^*$-algebra. Since $P$ is also $\sigma^*$-continuous, $\ran P$ is actually a W$^*$-algebra.
\end{proof}

In the case that $ P $ is \emph{faithful}, i.e., if $P ( x ) = 0$ for $x \in \BA _{ + }$ implies $x = 0$,
we obtain that $ \BA _{ r } $  is actually a W$^{*}$-\emph{sub}algebra of $ \BA $. 
To prove this, we need the following lemmas.

\begin{lemma}\label{lem:cp-mult-domain}
Let $ S\in \L(\A) $ be an operator satisfying $ S ( x ) ^{ * } S ( x ) \leq S ( x ^{ * } x ) $ for all
$ x \in \BA $.
Then
\[
  	\{ x \in \BA \colon S ( x ) ^{ * } S ( x ) = S ( x ^{ * } x ) \}
	= \{ x \in \BA \colon S ( x ) ^{ * } S ( y ) = S ( x ^{ * } y ) \; \text{ for all $ y \in \BA $} \}.
\]
\end{lemma}
\begin{proof}
Let $ \ph $ be a state on $ \BA $ and define for $ x , y \in \BA $ 
\[
	f _{ \ph }( x , y ) := \ph ( P ( y ^{ * } x ) - P ( y ^{ * } ) P ( x ) ).
\]
Then $ f _{ \ph } $ is a positive hermitian form on $ \BA $ for which
$ f _{ \ph } ( x , x ) = 0 $ if and only if $ f _{ \ph } ( x , y ) = 0 $ for all $ y \in \BA $
(see e.g. \cite[V.2., Prop. 2 and Cor. 1]{bourbaki-tvs}).
Since the states separate the points of $ \BA $, the assertion follows.
\end{proof}

\begin{prop}\label{prop:choi-effros-faithful}
If, in addition to the hypotheses of Theorem \ref{thm:main}, the semigroup $\S$ consists of completely positive operators and if $ P $ is faithful on $ \BA $, then $P$ is a conditional expectation, i.e., $ P ( y x z ) = y P ( x ) z $ for all $ y, z \in \BA_r $ and all $ x \in \BA $ and 
$ \BA _{ r } $ is a $W^{*}$-subalgebra of $ \BA $. 
\end{prop}
\begin{proof}
For every $ x \in \BA _{ r } $ we have $ P ( x ) = x $, hence $ P ( x ^{ * } ) = x ^{ * } $ by the positivity of $ P $ and 
\[
	x ^{ * } x = P ( x ^{ * } ) P ( x ) \leq P ( x ^{ * } x )
\]
by the Schwarz-inequality for completely positive maps.
Hence
\[
	0 \leq P ( P ( x ^{ * } x ) - x ^{ * } x )  = 0 
	\quad \text{and thus} \quad
	P ( x ^{ * } x ) = x ^{ * } x
\]
by the faithfulness of $ P $.
But then by Lemma \ref{lem:cp-mult-domain}
\begin{equation}\label{eq:claim}
	P ( x y ) = x P ( y ) 
	\quad \text{and} \quad
	P ( y x  ) = P ( y ) x
\end{equation}
for all $ x \in \BA _{ r } $ and $ y \in \BA $.
Hence $ P(xy)=xy$ for all $x,y\in\A_r$ and $\A_r$ is a W$^*$-subalgebra of $\A$.
\end{proof}

\begin{remark}
\label{rem:faithful}
  We note that if $S^*\ph=\ph$ for all $S\in\S$ and $\ph\in\Phi$, then the same holds for all elements in the closure of $\S$, thus in particular for $P$. If $\Phi$ is faithful, this implies that $P$ is faithful.
\end{remark}

\begin{lemma}\label{lem:schwarz-invertrierbar-automorphism}
Let $ S\in \L(\A) $ be an invertible operator satisfying $ S ( x ) ^{ * } S ( x ) \leq S ( x ^{ * } x ) $ 
and $ S ^{ - 1 }( x ) ^{ * } S ^{ -1 }( x ) \leq S ^{ - 1 }( x ^{ * } x ) $ for all
$ x \in \BA $.
Then $ S $ is a $^{*}$-automorphism on $ \BA $.
\end{lemma}
\begin{proof}
Notice that $S$ is positive by the first inequality. By the second inequality we have
\[
     S ^{ - 1 } \left( S ( x ) ^{ * } S ( x ) \right) \geq ( S ^{ - 1 } ( S x ^{ * } ) ) ( S ^{ - 1 } ( S  x ) )
     		= x ^{ * } x
\]
for all $x\in\BA$, and the positivity of $ S $ implies $S ( x ) ^{ * } S ( x ) \geq S ( x ^{ * } x )$.
Thus 
\[
     S ( x ^{ * } x ) = S ( x ^{ * } ) S ( x ) .
\]
Applying Lemma \ref{lem:cp-mult-domain} we then obtain $ S ( x ) ^{ * }S ( y ) = S ( x ^{ * } y ) $ for all 
$ x , y \in \BA $.
\end{proof}

\begin{proof}[Proof of Corollary \ref{cor:star_automorphisms}]

The operators $T_g$ are invertible on $\A_r$. Therefore, if the semigroup $\S$ consists of completely positive operators, then it follows from the Schwarz inequality and Lemma \ref{lem:schwarz-invertrierbar-automorphism} that every $ T _{ g } $ is a $^{*}$-automorphism.

Since the minimal projection $P$ is completely positive and faithful by Remark \ref{rem:faithful}, $\A_r$ is a W$^*$-subalgebra by Proposition \ref{prop:choi-effros-faithful}.
\end{proof}

If the semigroup is abelian, then the results in \cite{Albeverio} and \cite{olesen1980} yield a complete description of the reversible part and the action of $\S$ thereon.
\begin{thm} 
\label{prop:algebra-darstellung}
If the semigroup $\S$ is abelian and consists of completely positive operators and the action of $\S$ on $\A_r$ is ergodic, then there is a closed subgroup 
$ \mathbf{ H} $ of $ \GG $ and an ergodic action $ S $ of $ \mathbf{ H} $ on a factor $ \mathfrak{R} $ 
as $^{*}$-automorphisms such that 
\begin{equation}\label{eq:rotation-algebra}
	\BA _{r} \cong \{ f \in L ^{ \infty } ( \GG , \mathfrak{R} ) \colon
		f ( g h ) = S _{ h ^{-1} }( f ( g ) ) \; \text{for all $ g \in \GG, h \in \mathbf{H} $} \}
		\cong  L ^{ \infty } ( \GG/\mf{H})\otimes  \mathfrak{R}
\end{equation}
and $g$ corresponds to the map $ f \mapsto \lambda ( g ) f $ where 
$ ( \lambda ( g ) f )(s) = f ( g ^{ -1 } s ) $ for $s\in\GG$.

The factor $ \mathfrak{R} $ is actually the unique hyperfinite factor of type $ \mathrm{II} _{ 1 } $ or a matrix algebra $M_{n\times n}(\C)$.

\end{thm}

Note that if $\A$ is commutative, then $\mk{R}$ is trivial and we obtain the classical Halmos-von Neumann theorem
for which we refer to \cite[Theorem III.10.4 (Examples 3)]{Schaefer74} and the references cited there.

We now describe the stable part $\ker P$.

\begin{prop} The kernel of $P$ can be characterised as 
 $$\ker P=\{x\in\A \colon  0\in \SS x\}=\bigcup_{T\in\SS}\ker (T).$$
\end{prop}
\begin{proof}
 Clearly $\ker P\subset\bigcup_{T\in\SS}\ker T$. Now suppose $x\in\ker T$ for some $T\in\SS$. Since $PT\in M(\SS)$, we can find an $R\in\SS$ such that $R(PT)=P$ and thus
$Px=RPTx=0$.
\end{proof}

\section{Applications}\label{sect:ex}

\subsection{Perron-Frobenius for W$^*$-algebras}

We call a triple $ ( \BA , T , \varphi ) $ a \emph{W$^{*}$-dynamical system} if $ T $ is a completely positive operator on $ \BA $ with $ T \1 = \1 $ and $ T ^{ * } \varphi = \varphi $ for some faithful normal state 
$ \varphi $ on $ \BA $. We say that the W$^{*}$-dynamical system is \emph{ergodic} if the fixed space of $T$ is spanned by the unit of $\A$.

If $ ( \BA , T , \varphi ) $ is a W$^{*}$-dynamical system on the commutative W$^{*}$-algebra $ \BA = \C ^{ n } $, then the \emph{peripheral point spectrum} $ \Sp ( T ) \cap\mathbb{T} $ of $ T $ has an important symmetry yielding a certain asymptotic periodicity of the powers $ T ^{ n } $. 

Indeed, by the classical Perron-Frobenius theory there exists a natural number $ h $ such that 

\[
	\Sp ( T ^{ h } ) \cap \mathbb{T} = \{ 1 \}
\]

(see e.g. \cite[Theorem I.2.7]{Schaefer74}).

Then the powers of $ T ^{ h } $ converge to a projection $ P $ onto 

\[
	\Fix ( T ^{ h } ) = \ol{\lin}\{ x \in \BA \colon T x = \lambda x , \lambda \in \mathbb{ T }\} = \BA _{ r }
\]

with $ \ker P = \BA _{ s } $.

If $ S := T \circ P $, then 

\[
	\lim _{ n \to\infty} ( T ^{ n } - S ^{ n } ) = 0, \quad  S |_{  \BA _{ r } }^h = I |_{  \BA _{ r } },
			\quad S |_{ \BA _{ s } }=0.
\]

Thus the powers of $ T $ behave asymptotically as the ``periodic'' operator $ S $ which can be realised as a permutation matrix.
This has been generalised to positive operators on infinite-dimensional Banach lattices (see \cite{Schaefer74}) and on C$^{*}$- and W$^{*}$-algebras.

We show how parts of the results in \cite{ug1981,ug1984} follow naturally from the above JDLG-theory.

\begin{prop}[Non-commutative Perron-Frobenius]
\label{prop:non-commutative-Perron-Frobenius}
  Let $ ( \BA , T , \varphi ) $ be an ergodic W$^{*}$-dynamical system. Then the following assertions hold.
  \begin{enumerate}
  
  \item 
  The peripheral point spectrum of $ T $ is a subgroup of the circle group $ \mathbb{T} $.
  
  \item 
  Each peripheral eigenvalue $ \alpha $ of $ T $ is simple and $ \Sp( T ) = \alpha \cdot \Sp ( T )$.
  \item
  The $ \sigma ( \BA , \BA _{ * } ) $-closed subspace $ \BA  _{ r } $ generated by the eigenvectors pertaining 
  to the peripheral eigenvalues of $ T $ is a W$^{*}$-subalgebra of $ \BA $.
  \item
  The restriction of $ \varphi $ to $ \BA  _{ r } $ is a trace and there exists a completely 
  positive faithful projection $ P $ from $ \BA $ onto $ \BA _{ r } $ with 
  $ P ( y x z ) = y P ( x ) z $ for all $ x \in \BA $ and $ y, z \in \BA _{ r } $. 
  \end{enumerate}
\end{prop}
\begin{proof}
Because of $ \langle \varphi , (T x ) ^{ * } ( T x ) \rangle \leq \langle \varphi , T(x ^{ * } x) \rangle $ the semigroup  $ \S := \{ T ^{ n } \colon n \in \N \} $ fulfills the assumptions of Corollary \ref{cor:star_automorphisms}. This proves (3).

(4)
The projection $P$ is faithful by Remark \ref{rem:faithful} and thus the assertions follow from Proposition \ref{prop:unitary-eigenvector-trace} and
Proposition \ref{prop:choi-effros-faithful}.

(2)
Again by Proposition \ref{prop:unitary-eigenvector-trace} it follows that the eigenspaces pertaining to peripheral eigenvalues are one-dimensional, i.e., every peripheral eigenvalue is simple. 

If $ \alpha $ is a unimodular eigenvalue with unitary eigenvector $ u _{ \alpha } $, then we obtain 
\[
	T  ( u _{ \alpha } ^{ * } ) T ( u _{ \alpha } ) = 
	 u _{ \alpha } ^{ * } u _{ \alpha } = T ( u _{ \alpha } ^{ * } u _{ \alpha } )
\]
since $ T $ acts as $^{*}$-automorphism on $ \BA _{ r } $ by Corollary \ref{cor:star_automorphisms}. 
Thus for all $ x \in \BA $ 
\[
	T ( u _{ \alpha } x ) = \alpha u _{ \alpha } T ( x ) 		
\]
by Lemma \ref{lem:cp-mult-domain}, hence 
\[
	\alpha T ( x ) =  u _{ \alpha } ^{ * } T ( u _{ \alpha } x ).
\]
Since the map $ x \mapsto u _{ \alpha } x $ is a bijection on $ \BA $, we obtain
\[
	 \Sp( T ) = \alpha \cdot \Sp ( T ).
\]

(1)
If $ \alpha $ and $ \beta $ are unimodular eigenvalues with unitary eigenvectors $ u _{ \alpha } $ and 
$ u _{ \beta } $, then we obtain

\[
	T ( u _{ \alpha } u _{ \beta } ^{ * } ) = \alpha \overline{ \beta } u _{ \alpha } u _{ \beta } ^{ * }.
\]

Since  $ u _{ \alpha } u _{ \beta } ^{ * } $ is unitary, thus non-zero, we obtain that 
$ \alpha \overline{ \beta } $ is an eigenvalue and the peripheral point spectrum is a subgroup of the circle group.
\end{proof}

From Proposition \ref{prop:non-commutative-Perron-Frobenius} we deduce the following.

\begin{prop}\label{prop:lh=endliche-untergruppe}
Let  $ ( \L( H ) , T , \varphi ) $ be an ergodic W$^{*}$-dynamical system for some Hilbert space $ H $. Then the peripheral point spectrum of $ T $ is the cyclic group $ \Gamma _{ h } $ of all $ h $-th roots of unity for some $ h \in \N $.
\end{prop}
\begin{proof} 
Since $ \L ( H ) $ is an atomic \WA, the same is true for $ \BA _{ r } $ 
(see e.g. \cite[Exercise V.2 (8b)]{Takesaki}). 

Next we show that the centre $ \mathfrak{Z} $ of $ \BA _{ r } $ is finite-dimensional.
Taking this for granted, we obtain from \cite[Theorem V.1.17]{Takesaki} 
that $ \BA _{ r } $ is finite-dimensional and the peripheral eigenvalues form the finite cyclic subgroup $ \Gamma _{ h } $ of order $h$ of the circle group for some $ h \in \N $.

But if the centre $ \mathfrak{Z} $, which is atomic, is not finite-dimensional, then it is isomorphic to 
$ \ell ^{ \infty } $ and $ T $ induces an ergodic $ ^{ * } $-automorphism $ S $ on $  \mathfrak{Z} $ with
a faithful normal positive linear form $ \varphi _{ 0 } \in \ell ^{ \infty } _{ * } ( \cong \ell ^{ 1 } ) $ 
such that $ S ^{ * } ( \varphi _{ 0 } ) = \varphi _{ 0 } $.

But $ S $ is induced by some transformation $ \tau $ of $ \N $ onto $ \N $. 
In fact, if 
\[
  \delta _{ n } ( x ) = \xi _{ n } , \quad n \in \N, \; x \in  \mathfrak{Z}  , 
\]
then  $ \delta _{ n } \circ S $ is a normal, scalar valued $^{*}$-homomorphism with 
$ ( \delta _{ n } \circ S ) ( \1 ) = 1 $, thus of the form $ \delta _{ m } $ for some $ m = \tau ( n ) $;
thus $ S = S _{ \tau } $.

But since $ \tau $ is bijective, this conflicts with $ S ^{ * } ( \varphi _{ 0 } ) = \varphi _{ 0 } $ and the centre $ \mathfrak{Z} $ has to be finite-dimensional. 
\end{proof}
The next proposition follows from Theorem \ref{thm:main} because the minimal ideals of the generated semigroups on $ \BA _{ * } , \BA $ and on the GNS-Hilbert space $ \mathcal{H} _{ \varphi } $ are isomorphic.

\begin{prop}\label{prop:alle-spektren-sind-gleich}
Let $ ( \BA , T , \varphi ) $ be an ergodic W$^{*}$-dynamical system, let $ T _{ * } $ be the pre-adjoint operator of $ T $ on $ \BA _{ * } $ and let $ T _{ \varphi } $ be the extension of $ T $ to the GNS-Hilbert space $ \mathcal{H} _{ \varphi } $ pertaining to $ \varphi $.
Then the peripheral point spectra of $ T _{ * } , T $ and $ T _{ \varphi } $ coincide. 
\end{prop}

\subsection{Asymptotics}
Let $ ( \BA , T , \varphi ) $ be an ergodic W$^{*}$-dynamical system and take 
$ \A=\BA _{ r }\oplus \BA _{ s } $ to be the corresponding JDLG splitting of $ \BA $. 

The stable part $ \BA _{ s } $ can be characterised using Eisner \cite[Theorem II.4.1]{Eisner}, where the definition of the density of a subset of $\N$ can be found.
\begin{prop}
Let $ ( \BA , T , \varphi ) $ be an ergodic $ W^{*}$-dynamical system. 
Then the following assertions are equivalent.
\begin{enumerate}
\item
We have 
$ 0 \in 	\overline{\{ T x \colon T \in \SS \}} ^{ \sigma ( \BA , \BA_{*} ) }$ for all $ x\in \BA $.
\item
For every $x\in \A$ there exists a subsequence $ (n_j) _{ j \in \N} \subset \N $ with density 1 such that 
$ T ^ { n_j } x \xrightarrow[j\to\infty] {\s^*} 0$.
\item 
We have $\lim\limits_{n\to\infty}\frac{1}{n}\sum\limits_{k=0}^{n-1}|\sk{T^kx,\psi}|=0\quad \text{ for all } \;  \psi\in\BA_*$ and all $ x \in \BA $.
\end{enumerate}

Since $\A_*$ is separable, the following stronger condition is equivalent to the conditions above.
\begin{enumerate}
\item[(2')] There exists a subsequence $ (n_j) _{ j \in \N} \subset \N $ with density 1 such that 
$ T ^ { n_j } x \xrightarrow[j\to\infty] {\s^*} 0$ for all $x\in \A$.
\end{enumerate}

\end{prop}

We know that there exists a faithful conditional expectation $P$ from $ \BA $ onto $ \BA _{ r } $. If we define
$ S := T \circ P $, then we obtain the following.

\begin{cor}
Let $ ( \BA , T , \varphi ) $ be an irreducible  W$^{*}$-dynamical system  and let $ S $ be defined as above. 
Then there exists a subsequence $ (n_j) _{ j \in \N} \subset \N $ with density 1 such that 

\[
	 ( T ^ { n_j } - S ^{ n _{ j } } ) x \xrightarrow[j\to\infty] {\s^*} 0 
\]

for all $x\in \A$.

Thus $ T $ behaves ``almost weakly'' as an ergodic W$^{*}$-dynamical system  on $ \BA _{ r } $, and the 
structure of this algebra has been described in Proposition \ref{prop:algebra-darstellung}.
\end{cor}

\begin{remark}\label{rem:letzter-teil}
When $\A=\L(H)$ for some Hilbert space $H$, one even obtains convergence in the strong operator topology for the preadjoint semigroup $ \{ T _{ * } ^{ n } \colon n \in \N \} $ (see \cite[Theorem 3.3]{ug1984}). 
\end{remark}
\bibliographystyle{amsplain}

\end{document}